\documentclass[11pt]{amsart}
\usepackage{graphicx}
\usepackage{latexsym}
\usepackage{amsfonts,amsmath,amssymb}
\newtheorem{theorem}{Theorem}[section]
\newtheorem{proposition}[theorem]{Proposition}
\newtheorem{lemma}[theorem]{Lemma}

\newtheorem{corollary}[theorem]{Corollary}
\newtheorem{example}[theorem]{Example}

\newtheorem{fact}[theorem]{Fact}

\def\R{\mathbb{R}}

\author[M. B\'ona]{Mikl\'os  B\'ona}
\title[$k$-Protected Vertices in Binary Search Trees]{$k$-Protected vertices in Binary Search Trees}
\address{\rm M. B\'ona, Department of Mathematics, 
University of Florida,
358 Little Hall, 
PO Box 118105, 
Gainesville, FL 32611--8105 (USA)
}

\date{\today}

\begin{document}

\begin{abstract}
We show that for every $k$, the probability that a randomly selected vertex of 
a random binary search tree on $n$ nodes is at distance $k-1$ from the closest leaf converges to
a rational constant $c_k$ as $n$ goes to infinity. 
\end{abstract}

\maketitle

\section{Introduction}
\subsection{2-Protected vertices in trees}  A {\em 2-protected vertex}  in a rooted tree is a vertex that
is not a leaf and is not adjacent to a leaf. In social networks, protected vertices may represent 
participants who have, in the past, invited others to join the network, but have not recently done that. 
This, and other applications led to a recent flurry of interest in studying 2-protected vertices in 
various kinds of rooted trees. See the articles \cite{cheon}, \cite{du} and \cite{ward} for some results.

\subsection{Vertices at level $k$} We generalize the notion of 2-protected vertices as follows. 
In a rooted tree, we say that vertex $v$ is {\em at level} $k$, or is $(k-1)$-{\em protected} if the shortest path from $v$ to any leaf
of the tree consists of $k-1$ edges. In other words, the distance between $v$ and the closest leaf
is $k-1$. So leaves are at level 1, neighbors of leaves are at level two,
and so  on. In particular, 2-protected vertices are those that are at level 3 or higher. 

In this paper, we will  study the numbers of vertices at level $k$ in {\em binary search trees}, which are sometimes also called {\em decreasing binary
trees}, and which are in one-to-one correspondence with permutations as explained below. 

Let $p=p_1p_2\cdots p_n$ be a permutation. The {\em binary search tree} of $p$, which
we denote by $T(p)$, is defined as follows. The root of $T(p)$ is a vertex labeled $n$, the largest entry of $p$. 
If $a$ is the largest entry of $p$ on the left of $n$, and $b$ is the largest
entry of $p$ on the right of $n$, then the root will have two children,
the left one will be labeled $a$, and the right one labeled $b$. If $n$ is
the first (resp. last) entry of $p$, then the root will have only one child,
and that is a left  (resp. right) child, and it will necessarily be
labeled $n-1$ as $n-1$ must be the largest of all remaining elements.
Define the rest of $T(p)$ recursively,  by taking $T(p')$ and $T(p'')$, where
$p'$ and $p''$ are the substrings of $p$ on the two sides of $n$, and affixing
them to $a$ and $b$.

Note that $T(p)$ is indeed
a binary tree, that is, each vertex has 0, 1, or 2 children. Also note  that each child
 is a left child or a right child of its
parent, even if that child is an only child.  Given $T(p)$, 
we can easily recover $p$ by reading $T$ according to the tree
traversal method called {\em in-order}.
In other words,
first we read the left subtree of $T(p)$, then the root, and then the right
subtree of $T(p)$. We read the subtrees according to this very same rule.
See Figure \ref{permtree} for an illustration. 

\begin{figure} \label{permtree}
\begin{center}
  \includegraphics[width=60mm]{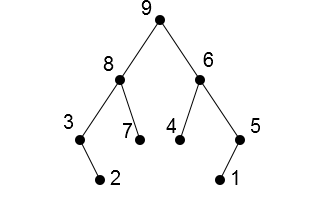}
  \caption{The tree $T(p)$ for $p=328794615$.}
  \end{center}
 \end{figure}
 
 Because of this one-to-one correspondence between permutations and binary search trees,
 in our discussion, we will use these two kinds of objects interchangeably. 
 
 As a warmup, we try a simple probabilistic approach, which will only be successful in the cases of 
 $k=1$ and $k=2$. In the case of general $k$, it will provide only a rough lower bound, but
 that lower bound will be useful in the following section. In that section, we use an analytic
 approach which, in theory, provides the exact form of the exponential generating function
 $A_k(x)$ of the total number of vertices at level $k$ in all binary search trees. In practice, 
 these generating functions will have a large number of summands. However, we will be able
 to describe them in sufficient precision to find the growth rate of their coefficients. 
 
\section{Warm-up: A  Probabilistic Approach}
\subsection{Two Simple Initial Cases}
In this section, we enumerate vertices on levels one and two. It turns out that these
cases are much simpler than the general case, and do not necessitate the general method
that we will use in later sections. 

In order to alleviate notation, let us agree that for the rest of this paper, all permutations are 
of length $n$. Let $X(p)$ denote the number of leaves in the tree $T(p)$, and let
$E(X)$ denote the expectation of $X(p)$ taken over all permutations $p$ of length $n$. 

\begin{proposition} \label{level1} For all  integers $n\geq 2$, the equality $E(X)=\frac{n+1}{3}$ holds.
\end{proposition}

\begin{proof} Let $p=p_1p_2\cdots p_n$.  Let $2\leq i\leq n-1$. Then it is straightforward to prove,
for instance by induction on $n$, that the vertex corresponding to  $p_i$ is a leaf if and only
if it is smaller than both of its neighbors, and that event  has probability $1/3$. On the other hand, 
if $i\in\{1,n\}$, then $p_i$ is a leaf if and only if it is smaller than its only neighbor, an event of
probability $1/2$, Therefore, if we denote by $X_i(p)$ the indicator variable of the
event  that $p_i$ is a leaf, then 
by linearity of expectation we get
\[E(X)=\sum_{i=1}^n E(X_i) =(n-2)\cdot \frac{1}{3} + 2\cdot \frac{1}{2}=\frac{n+1}{3}.\]
\end{proof}

It is perhaps a little bit surprising that the formula for entries at level two is just as simple
as the formula proved in Proposition \ref{level1}. Let $Y(p)$ denote the number of vertices
of $p$ that are at level two. 

\begin{theorem} \label{level2}
Let $n\geq 4$. Then the equality $E(Y)=\frac{3(n+1)}{10}$ holds. 
\end{theorem}

\begin{proof} 
Let $a_{n,2}$ be the total number of vertices in all decreasing binary trees on $n$ vertices
that are at level two. Note that if $n>1$, then each leaf must have a unique parent, and that
parent must always be a vertex at level two. However, some vertices at level two are parents of
{\em two} leaves. We will now determine the number $d_n$ of such vertices, which will
then yield a formula,
\begin{equation} \label{sieve} a_{n,2}=\frac{(n+1)!}{3} - d_n\end{equation}
for $a_{n,2}$, where $n\geq 4$. 

Let $p_i$ be a vertex that is at level two and has two leaves as children. Let us assume for now
that $3\leq i \leq n-2$ holds. Then $p_i$ is larger than both of its neighbors, and both of those 
neighbors $p_{i-1}$ and $p_{i+1}$ are leaves, so they are smaller than both of their neighbors,
meaning that $p_{i-1}<p_{i-2}$, and $p_{i+1}<p_{i+2}$. On the other hand $p_i$ must be smaller
than both of its second neighbors, otherwise its children could not be $p_{i-1}$ and $p_{i+1}$.
This means that if out of the 120 possible permutations of the mentioned five entries, 
only four are possible, since $p_i$ must be the middle one in size, its neighbors must be 
the two smallest entries, and its second neighbors must be the two largest entries. 
So if $Z_i(p)$ is the indicator variable of the
event  that $p_i$ has two leaves as children (in which case $p_i$ is necessarily at 
level two), then for $i\in [3,n-2]$, we get
$E(Z_i)=\frac{4}{120}=\frac{1}{30}$. If $i=1$ or $i=n$, then $p_i$ cannot have two children.
Finally, if $i=2$ or $i=n-1$, then an analogous argument shows that $E(Z_i)=\frac{2}{24}=
\frac{1}{12}$.
Therefore, since $Z=\sum_{i=2}^{n-1} Z_i$ denotes the number of vertices that have two leaves as children (and are therefore at level two), then by linearity of expectation we have
\[E(Z)=\sum_{i=2}^{n-1} E(Z_i)=2\cdot \frac{1}{12} + (n-4) \cdot \frac{1}{30} =\frac{n+1}{30}.\]
Therefore, $d_n=(n+1)!/30$, so formula (\ref{sieve}) implies that 
\[a_{n,2}=\frac{(n+1)!}{3} -\frac{(n+1)!}{30} =\frac{3}{10}\cdot (n+1)!,\]
which proves our claim. 
\end{proof}

As a vertex in a rooted tree is called 2-protected if it is not at level 1 or 2, we can now easily
compute the expected number $E(Prot_n)$ of protected vertices in  binary search trees of size $n$.
We recover the following result of Mark Ward and H. Mahmoud \cite{ward}.
\begin{corollary} For $n\geq 4$, the equality
\[E(Prot_n)= \frac{11n-19}{30}\]
holds.
\end{corollary}


\subsection{Higher values of $k$}
If $k>2$, then finding the total number of vertices at level $k$ is significantly more complicated.
The main reason for this is that if $k>2$, then the unique parent of a vertex at level $k-1$ does not
have to be a vertex at level $k$; it can be a vertex at level $\ell$, where $1<\ell\leq k$. 
For instance, in the tree $T(p)$ shown in Figure \ref{permtree}, vertex 3 is at level two, and its
parent, vertex 8, is also at level two. 

\subsubsection{A simple, but useful Lemma}
Let $a_{n,k}$ be the total number of vertices at level $k$ in all decreasing binary trees at level 
$k$. It is then clear that $a_{n,k+1} \leq a_{n,k}$ since each vertex at level $k+1$ must have
at least one child at level $k$. While finding the exact value of $a_{n,k}$ is beyond the 
scope of this introductory section, the following lemma will turn out to be useful for us, even if
its bound is far from being optimal. 

\begin{lemma} \label{positive} For each positive integer $k$, there exists a positive constant $\gamma_k$ so that if $n$ is large enough, then 
\[\frac{a_{n,k}}{ n \cdot n!} \geq  \gamma_k.\] 
\end{lemma} 
In other words, for any fixed $k$, the probability that a randomly selected vertex of 
a randomly selected decreasing binary tree of size $n$ is at level $k$  is larger than $\gamma_k$.

Before we prove lemma \ref{positive}, we need a simple notion. 
A {\em perfect binary tree} is a binary tree in which every non-leaf vertex has two
children, and every leaf is at the same distance from the root. So a perfect binary tree in which the root is at level $\ell$ has $1+2+\cdots +2^{\ell-1}=2^{\ell}-1$ vertices.  

We will now compute the expected number of vertices $p_i$ that are at level $k$  for which the subtree rooted at $p_i$  is a perfect binary tree. The expected number of such vertices is 
obviously a lower bound for the expected number of vertices at level $k$. 

Let $Q_k$ be the probability that for a randomly selected permutation $p$ of length
$2^k-1$, the tree $T(p)$ is a perfect binary tree (disregarding the labels). It is then clear
that $Q_1=1$, and 
\begin{equation} \label{pbintree}
Q_{k+1} = \frac{1}{2^{k+1}-1} Q_k^2.
\end{equation}
So $Q_2=1/3$, and $Q_3=1/63$.  In particular, $Q_k$ is always a positive real number.

\begin{proposition} \label{perfect}
Let $p=p_1p_2\cdots p_n$ be a permutation, and let $2^{k-1}+1 \leq i\leq n-2^{k-1}$. (In other
words, $i$ is not among the smallest $2^{k-1}$ indices or the largest $2^{k-1}$ indices in $p$.)
Let $P_k$ be the probability that the vertex $p_i$ of $T(p)$  is at height $k$, and the
 subtree of $T(p)$ rooted  $p_i$
is a perfect binary tree. Then the equation
\[P_{k}=Q_k\cdot \frac{2}{(2^k+1)2^k} \]
holds for $k\geq 1$. In particular, $P_k$ is a positive real number that does not depend on $n$. 
\end{proposition}

\begin{proof}
The subtree rooted at the vertex $p_i$ of $T(p)$ will be a perfect binary tree with its root at level
$k$ if the following two independent events occur.
\begin{enumerate}
\item The string $p_{[i,k]}$of $2^{k}-1$ consecutive entries of $p$ whose middle entry is $p_i$ 
correspond to a binary search tree that is a perfect binary tree, and 
\item all entries in $p_{[i,k]}$ are less than {\em both} entries bracketing $p_{[i,k]}$, that is, both
$p_{i-2^{k-1}}$ and $p_{i+2^{k-1}}$. 
\end{enumerate}
The first of these events occurs at probability $Q_k$, and the second one occurs at
probability $\frac{2}{(2^k+1)\cdot 2^k}$, proving our claim.
\end{proof}
So $P_1=1/3$, and $P_2=1/30$, as we computed in the proofs of Proposition \ref{level1} and
Theorem \ref{level2}.  Furthermore, 
\[P_3=Q_3\cdot \frac{2}{8\cdot 9}=\frac{1}{63}\cdot \frac{1}{36}=\frac{1}{2268}.\]
Now we are in a position to prove Lemma \ref{positive}. 

\begin{proof} (of Lemma \ref{positive}) 
Let $V_i(p)$ be the indicator random variable of the event that the subtree of $p$ that is 
rooted at $p_i$ is a perfect binary tree whose root is at level $k$. Then it follows from the 
definition of $P_k$ that
\[E(V_i(p))=P_k. \]
If $V(p)$ denotes the number of vertices of $p$ that are at level $k$ and whose subtrees are 
perfect binary trees, then the  linear property of expectation yields 
\[E(V(p))= (n-2^{k})P_k, \] since we do not allow $i$ to be among the smallest $2^{k-1}$ indices
or the among the largest $2^{k-1}$ indices. 
Therefore, the total number $a_{n,k}$  of vertices at level $k$ in all decreasing binary trees of size 
$n$
satisfies \[\frac{a_{n,k}}{n\cdot  n!} \geq \left(1-\frac{2^{k}}{n}\right)P_k \geq \frac{P_k}{2}\]
for $n\geq 2^{k+1}$. 
This completes the proof, since we can set $\gamma_k=P_k/2$. 
\end{proof}

\section{The Analytic Approach}
\subsection{A System of Differential Equations}

In order to determine the exact value of $a_{n,k}$ for $k\geq 3$, we turn to exponential generating functions. We recall the well-known fact that
the exponential generating function for the number of permutations of length $n$, and
equivalently, decreasing binary trees on $n$ vertices, is $\sum_{n\geq 0} n!\frac{x^n}{n!} =1/(1-x)$.

For $k\geq 1$, let $A_k(x)$ denote the exponential generating function of the numbers of 
all vertices at level $k$ in all decreasing binary trees of size $n$.  Let $B_k(x)$
denote the exponential generating function for such trees in which the root is at level $k$.
In both $A_k(x)$, and $B_k(x)$, we set the constant term to 0. 
Note that $A_k(x)=\sum_{n\geq 1} \frac{a_{n,k}}{n!}x^n$, so in particular, the coefficient of
$x^n$ in $A_k$ is the {\em expected number} of vertices at level $k$ in a randomly selected
decreasing binary tree of size $n$.

Then the following differential equations hold. 

\begin{lemma} \label{lemmaforb}
We have $B_1(x)=x$, and 
\[B_k'(x)=2B_{k-1}(x)\cdot \left (\frac{1}{1-x} - B_1(x)-B_2(x)-\cdots -B_{k-2}(x) \right) -B_{k-1}(x)^2\]
if $k>1$.
\end{lemma}

\begin{proof}
Let $T$ be a binary search tree counted by $B_k(x)$. Let us remove the root of $T$. On the one hand, this yields
a structure counted by $B_k'(x)$. On the other hand, this yields an ordered pair of binary search
trees such that one of them has its root at level $k-1$, and the other one has its root 
at level $k-1$ or higher. By the Product Formula for exponential generating functions (see for 
instance, Chapter 8 of \cite{walk}), such pairs
are counted by the first product on the right-hand side.  At the end of the right-hand side, we must subtract $B_{k-1}(x)^2$
as ordered pairs in which {\em both} trees have their root at level $k-1$ are double-counted by
the preceding term.
\end{proof}

\begin{example} \label{exB2} Let $k=2$. Then Lemma \ref{lemmaforb} yields
\begin{eqnarray*} B_2'(x) & = & 2B_1(x)\cdot \left(\frac{1}{1-x}\right) - B_1(x)^2 \\
& = & \frac{2x}{1-x} - x^2.
\end{eqnarray*}
Therefore, using the equality $B_2(0)=0$, we deduce that 
\[B_2(x)=2\ln \left(\frac{1}{1-x}\right)  - 2x -\frac{x^3}{3} .\]
\end{example}

\begin{lemma}  \label{diffeqfora} For $k\geq 1$, the linear differential equation 
\[A_k'(x)=\frac{2}{1-x}\cdot A_k(x) + B_k'(x)\]
holds. 
\end{lemma}

\begin{proof}
Let $(T,v)$ be an ordered pair so that $T$ is a binary search  tree on $n$ vertices,
and $v$ is a vertex of $T$ that is at level $k$. Now remove the root of $T$. If the root was $v$
itself, then we get a structure counted by $B_k'(x)$, just as we did in the proof of Lemma
\ref{lemmaforb}. Otherwise, we get an ordered pair $(R,S)$ of structures, one of which is
a binary search tree, and the other one of which is an ordered pair of a  binary search tree
and a vertex of that tree that is at level $k$. This explains the first summand of the right-hand side
by the Product Formula. 
\end{proof}

\begin{example}
Setting $k=2$, we see that $A_2(x)$ is the unique solution of the linear differential equation
\[A_2'(x)=\frac{2}{1-x}\cdot A_2(x) +   \frac{2x}{1-x} - x^2-2\]
with initial condition $A_2(0)=0$. This yields
\[A_2(x)=\frac{-\frac{1}{5}x^5+\frac{1}{2}x^4-x^3+x^2}{(1-x)^2}.\]
\end{example}

\section{A class of functions, and needed facts about integration}
In this section, we define a class of functions that will be useful to describe our results.

\subsection{A class of functions}
Let ${\bf PL}(x)$ be the class of functions $f:\R\rightarrow \R$ which are of the form 
\begin{equation} \label{classdef} 
f(x)=\sum_{i=1}^{m}  a_i  (1-x)^{b_i} \ln \left ( \frac{1}{1-x} \right)^{c_i},\end{equation}
where  the coefficients $a_i$ are  rational numbers, while the exponents $b_i$ and $c_i$ are non-negative integers. Roughly speaking, ${\bf PL}(x)$ is the class of functions that are "polynomials
in $1-x$ and $\ln \left ( \frac{1}{1-x} \right)$". 


A few facts about ${\bf PL}(x)$ that are straightforward to prove using integration by parts will 
be useful in the next section. We do not want to break the course of the discussion for such
technicalities, and therefore we will present them in the Appendix. 

\subsection{The general form of $A_k(x)$ and $B_k(x)$}

Now we are in a position to determine the general form of $A_k(x)$ and $B_k(x)$ with sufficient
precision to  deduce the asymptotic number of all entries at level $k$ in all permutations of length 
$n$. 
We start with $B_k(x)$.

\begin{lemma} \label{Bisgood}
For all $k\geq 1$, we have
\[B_k(x)\in {\bf PL}(x). \]
\end{lemma}

\begin{proof}
We prove the statement by induction on $k$. It is obvious that $B_1(x)=x$, and we saw in Example
\ref{exB2} that $B_1(x)=x$ and $B_2(x)=2\ln \left(\frac{1}{1-x}\right)  - 2x -\frac{x^3}{3} $. So the statement is true for $k=1$ and $k=2$. Now let us assume that the statement of 
the lemma holds for all positive integers less than $k$. It then follows from Lemma \ref{lemmaforb}
that the summands of $B_k'(x)$ are all in ${\bf PL}(x)$, except possibly some summands of the form 
$a_i \cdot \frac{1}{1-x} \cdot  \ln  \left(\frac{1}{1-x}\right)^{c_i} $, where $a_i$ is a rational number and
$c_i$ is a non-negative integer. The integral of each such  summand
is in ${\bf PL}(x)$ by Fact \ref{easy}, and integrals of the other summands (those that
are in ${\bf PL}(x)$) are in ${\bf PL}(x)$ by Proposition \ref{closed}. Therefore, as
${\bf PL}(x)$ is closed under addition, our claim is proved. 
\end{proof}

While the power series $A_k(x)$ are in general not in ${\bf PL}(x)$, the following
weaker statement does hold for them. 

\begin{theorem} \label{formofA} For all positive integers $k$, we have 
\begin{equation} \label{eqfora} A_k(x)=\frac{p_k(x)}{(1-x)^2} +f(x),\end{equation}
where $f(x) \in {\bf PL}(x)$, and $p_k(x)$ is a polynomial function with rational coefficients that is not divisible
by $(1-x)$.
\end{theorem}

\begin{proof}
Lemma \ref{diffeqfora} provides a linear differential equation for $A_k(x)$. Solving that equation, we
get \begin{equation}
A_k(x)=\frac{\int B_k'(x) (1-x)^2 \ dx}{(1-x)^2} + \frac{C}{(1-x)^2},
\end{equation}
where the integral on the right-hand side is meant with 0 as constant term. 

We saw in the proof of Lemma \ref{Bisgood} that the summands of $B_k'(x)$ are all in 
${\bf PL}(x)$,  except possibly some summands of the form 
$a_i \cdot \frac{1}{1-x} \cdot  \ln  \left(\frac{1}{1-x}\right)^{c_i} $. Therefore, the summands of 
$(1-x)^2B_k'(x)$ are all in ${\bf PL}(x)$. 
Even more strongly, each summand of $(1-x)^2B_k'(x)$ is of the form 
$a_i(1-x)^{b_i}\ln \left( \frac{1}{1-x} \right)^{c_i}$, with $b_i \geq 1$. Therefore, Proposition
\ref{refined} implies that the integral of each summand is of the form 
$(1-x)^{b_i+1}g_i(x) + p_{\langle i\rangle}(x)$, where $g_i(x)\in {\bf PL}(x)$, and
 $p_{\langle i\rangle}(x)$ is a polynomial 
function with rational coefficients. As $b_i+1 \geq 2$, this implies that $\int B_k'(x) (1-x)^2 \ dx =(1-x)^2g(x)+ q_k(x)$, 
where $g(x)\in {\bf PL}(x)$ and $q_k(x)$ is a polynomial function with rational coefficients, and our claim is proved.  
\end{proof}

In other words, though $A_k(x)$ contains terms in which $(1-x)^2$ is in the denominator, those
terms are simply rational functions; they do not contain logarithms. 
This is important since the coefficients of the power series \[\frac{\ln(1/(1-x))}{(1-x)^2}\]
grow faster than those of the terms that occur in $A_k(x)$. (In fact, their growth is faster than
linear.)

We are now ready to state and prove the main result of this paper. 

\begin{theorem} \label{main}
Let $k\geq 1$, and let $a_{n,k}$ be the number of all vertices at level $k$ in all 
binary search trees on $n$ vertices. Then there exists a rational constant $c_k$
so that \[\lim_{n\rightarrow \infty} \frac{a_{n,k}}{(n+1)!} =c_k .\]
\end{theorem}

\begin{proof}
Let $[x^n]H(x)$ denote the coefficient of $x^n$ in the power series $H(x)$. 
Formula (\ref{eqfora}) shows the general form of $A_k(x)$.
The second summand on the right-hand side of (\ref{eqfora}) is a function 
$f\in {\bf PL}(x)$. Each summand of $f$ is of the form $a_i(1-x)^{b_i} \left(\ln \frac{1}{1-x}\right)^{c_i} $,
where, crucially, $b_i\geq 0$ and $c_i\geq 0$, while the $a_i$ are rational numbers. 

 It is proved in Theorem VI.2. of 
{\em Analytic Combinatorics} \cite{analco}, in particular in formula (27) on page 386, that 
if $b_i\geq 0$ and $c_i> 0$, then 
\begin{equation} \label{tflaj}
[z^n] \left ((1-x)^{b_i} \ln \left(\frac{1}{1-x}\right)^{c_i}\right)\sim n^{-b_i-1}\sum_{j\geq 0} \frac{F_j (\ln n)}{n^j} ,\end{equation}
where the $F_j$ are constants. 

In particular, in each summand of $f(x)$, the coefficient of $x^n$ is less than $K(\ln n)/n$ for some constant $K$, and as such, 
it is  negligibly small compared to $n$.   Therefore,  the contribution of $f(x)$ to   $[x^n]A_k(x)$ is negligible,
since we know from Lemma \ref{positive} that $[x^n]A_k(x) \geq \gamma_k \cdot n$ for
a positive constant $\gamma_k$. 

Now we turn to  the first summand of formula (\ref{eqfora}) for $A_k(x)$. 
This summand, $\frac{p_k(x)}{(1-x)^2}$ is simply a rational function. Its numerator, 
$p_k(x)$ cannot be divisible by $(1-x)$, since that would imply that the coefficients 
of $x^n$ in $\frac{p_k(x)}{(1-x)^2}$ are all smaller than a constant. (See for instance
Theorem IV.9 in \cite{analco}.) That would be a contradiction since we know from 
Lemma \ref{positive} that $[x^n]A_k(x)\geq \gamma_k n$, and the result of the previous 
paragraph implies that that means $[x^n]\frac{p_k(x)}{(1-x)^2} \geq \gamma_k' n$.

So  the rational function $\frac{p_k(x)}{(1-x)^2}$ has a pole of order two at 1. It is now routine
to prove (see again Theorem IV.9 in \cite{analco} or the discussion that follows here) that 
\begin{equation} \label{firstterm}
[x^n]\frac{p_k(x)}{(1-x)^2}\sim c_k (n+1) \end{equation} for some constant $c_k$.
As we have seen that the contribution of $f(x)$ in  (\ref{eqfora}) is insignificant,  (\ref{eqfora}) and (\ref{firstterm}) together imply that
\[[x^n]A_k(x) \sim [x^n]\frac{p_k(x)}{(1-x)^2}\sim c_k (n+1).\]
Finally, we prove that $c_k$ is rational. In order to see this, note that if $n$ is large enough
then \[[x^n]\frac{p_k(x)}{(1-x)^2} =[x^n]\frac{ax+b}{(1-x)^2},\]
where $ax+b$ is the remainder obtained when $p_k(x)$ is divided by $(1-x)^2$. As $p_k$ has
rational coefficients, both $a$ and $b$ are rational numbers. 
However, 
\[[x^n]\frac{ax+b}{(1-x)^2}=[x^n]\left(\frac{a+b}{(1-x)^2} - \frac{a}{(1-x)}\right) =(a+b)(n+1)-a,\]
so $c_k=a+b$, which is a rational number. 
\end{proof}

\section{Examples}
\subsection{The case of $k=3$.}
Determining the value of $c_3$ requires finding $B_3'(x)$ first. We can do that by
using Lemma \ref{lemmaforb}, since $B_1(x)$ and $B_2(x)$ have already been computed
in Lemma \ref{lemmaforb} and Example \ref{exB2}. A routine computation that we carried out using Maple leads
to 
\begin{eqnarray} \label{eqb3}
B_3'(x) & = & 4\frac{\ln(1/(1-x))}{1-x}+4x\ln(1/(1-x))-\frac{2}{3}\frac{x^3}{1-x}-\frac{2}{3}x^4 \\
& - &
4\frac{x}{1-x}  - 4\ln(1/(1-x))^2  +
 \frac{4x^3}{3}\ln(1/(1-x))-\frac{x^6}{9}.
\end{eqnarray}

Now we can solve the differential equation provided by Lemma \ref{diffeqfora} with $k=3$, to get
\begin{eqnarray*}
A_3(x) & = & \frac{1721}{8100(1-x)^2)}-\frac{x^7}{81}+\frac{x^6}{324}-\frac{5x^5}{54}+
\frac{2x^4}{9}\ln(1/(1-x))  + \frac{23x^4}{324} \\ 
& - & \frac{4x^3}{45}\ln(1/(1-x))+
\frac{349x^3}{2025} +\frac{14x^2}{15}\ln(1/(1-x))+\frac{979x^2}{2700} \\
& - & \frac{8x}{5}\ln(1/(1-x)) 
+ \frac{4219x}{4050}-\frac{4x}{3}\ln(1/(1-x))^2 \\ 
 & + & \frac{4}{3}\ln(1/(1-x))^2-\frac{1721}{8100} 
-\frac{22}{15}\ln(1/(1-x)).
\end{eqnarray*}

It is now clear, by the proof of Theorem \ref{main} that 
\[c_3=\lim_{n\rightarrow \infty} \frac{[x^n]A_3(x)}{(n+1)!}=\frac{1721}{8100}\sim 0.2124691358.\]

\subsection{The case of $k=4$.}
Determining the value of $c_4$ is conceptually the same as determining $c_3$.
However, the computation becomes much more cumbersome. Lemma \ref{lemmaforb} provides
a formula for $B_3'(x)$ as a sum. According to Maple, that sum has 52 summands of the
form $a_ix^{b_i}\left(\ln(1/(1-x))\right)^{c_i}$. Using that expression for $B_3'(x)$, we can compute
$A_4(x)$ using Lemma \ref{diffeqfora}. Maple obtains a solution that has 59 summands. 
However, only 17 of these 59 summands contribute to $p_4(x)$, and therefore, only these
17 summands influence $c_4$. 
The value we obtain for $c_4$ is 
\[c_4=\frac{250488312501647783}{2294809143026400000}=0.1091543117.\]

\section{Further Directions}
The data that we computed, $c_1=1/3$, $c_2=0.3$, $c_3=0.212$, and $c_4=0.109$ suggest
that the sequence $c_1,c_2,\cdots $ is log-concave. Is that indeed the case, and if so,  is there a 
combinatorial proof? We point out that it is not true that for any fixed $n$, the 
sequence $a_{n,1}, a_{n,2},\cdots ,a_{n,n}$ is log-concave. For instance, $n=4$ provides a
counterexample. However, it can still be the case that for every $k$, there exists a threshold
$N(k)$ so that if $n>N(k)$, then the sequence $a_{n,1},a_{n,2},\cdots ,a_{n,k}$ is log-concave. 

\section{Appendix:Needed facts about integration}
\begin{proposition} \label{closed}
The class ${\bf PL}(x)$ is closed under integration with respect to $x$. 
\end{proposition}

\begin{proof}
We need to show that $\int (1-x)^{b} \ln \left ( \frac{1}{1-x} \right)^{c} \ dx \in {\bf PL}(x)$.
We prove this by induction on $c$, the inital case of $c=0$ being obvious. 
Integration by parts yields
\begin{eqnarray}  \label{partint} 
\int (1-x)^{b} \ln \left ( \frac{1}{1-x} \right)^{c} \ dx & = & - \ln \left ( \frac{1}{1-x} \right)^{c}  \cdot 
\frac{(1-x)^{b+1}}{b+1}  \\ & + & \int \frac{(1-x)^{b}}{b+1} \cdot c \ln \left ( \frac{1}{1-x} \right)^{c-1} \ dx . \end{eqnarray}
By the induction hypothesis, the integral on the right-hand side is in ${\bf PL}(x)$, proving our claim. 
\end{proof}

A special case of the previous proposition will be particularly useful for us. 
\begin{proposition} \label{refined}
Let  $b\geq 0$ and $c\geq 0$ be integers. Then 
\[\int (1-x)^{b} \ln \left ( \frac{1}{1-x} \right)^{c} \ dx = (1-x)^{b+1} \cdot g(x) + p(x) ,\]
where $p$ is a polynomial function, and $g(x)\in {\bf PL}(x)$. The integral on the
left-hand side is taken with constant term 0. 
\end{proposition}

\begin{proof}
Induction on $c$, the initial case being  that of $c=0$. If $b=0$, then the statement is true, 
since $\int 1 \ dx = x= (1-x) \cdot (-1) + 1$. If $b>0$, then the statement is true, since
$\int (1-x)^b \ dx = (1-x)^{b+1} \cdot \frac {-1}{n+1}$. 
 
 The induction step directly follows
from (\ref{partint}) and from the induction hypothesis.
\end{proof}
We will also need the following. 

\begin{fact} \label{easy}
For all non-negative integers $c$, the equality 
\[ \int \frac{1}{1-x} \cdot \ln \left ( \frac{1}{1-x} \right)^{c} dx =  \frac{1}{c+1} \ln \left ( \frac{1}{1-x} \right)^{c+1} +C\]
holds.  
\end{fact}

\end{document}